\documentclass[11pt]{article}
\usepackage{amsmath,amstext,amssymb,amsthm}

\usepackage[english]{babel}
\usepackage[T1]{fontenc}
\usepackage[utf8]{inputenc}
\usepackage{tikz}
\usepackage{url}
\usepackage{times}

\usepackage{hyperref}

\newtheorem{theorem}{Theorem}
\newtheorem{lemma}[theorem]{Lemma}

\theoremstyle{definition}
\newtheorem{example}{Example}

\newtheorem{problem}{Problem}

\newcommand{\Hall}{\mathbf{m}}
\newcommand{\Al}{\Sigma}

\newcommand{\eps}{\varepsilon}
\newcommand{\per}{\partial}

\newcommand{\EXTRA}[1]{}

\newcommand{\Cr}{\eta}
\newcommand{\sdot}{{\cdot}}
\begin{document}

\title{Critical Factorisation in Square-Free Words}

\author{Tero Harju\\
Department of Mathematics and Statistics\\
         University of~Turku, Finland\\
         \texttt{harju@utu.fi}\\
         }

\maketitle

\begin{abstract}
A position $p$ in a word $w$ is critical if the minimal local period at $p$ 
is equal to the global period of $w$. According to the Critical Factorisation Theorem
all words of length at least two have a critical point. We study the number 
$\Cr(w)$ of critical points of square-free ternary words $w$, i.e.,  words over a three letter alphabet. 
We show that the sufficiently long square-free words $w$ satisfy 
$\Cr(w) \le |w|-5$ where $|w|$ denotes the length of $w$. Moreover, the bound $|w|-5$ is reached by
infinitely many words.
On the other hand, every square-free word $w$ has at least $|w|/4$ critical points, and there is a sequence of these words closing to this bound.
\end{abstract}


\section{Introduction}

The Critical Factorisation Theorem~\cite{CesariVincent:78,Duval:79}
is one of the gems in combinatorics on words. It
states that each word $w$ with $|w| \ge 2$
has a critical point, i.e., a position where the local period $\per(w, p)$ is equal to 
the global period $\per(w)$ of the word. For a word $w$
with a factorisation 
$w=xy$, $\per(w, |x|)$ denotes the length of the shortest word $u$ such that of $u$ and $x$ one is a suffix of the other,
and of $u$ and $y$ one is a prefix of the other.

In the binary case, say $w \in \{0,1\}^*$, it was shown 
in~\cite{HarjuNowotka} that there are words having only one
critical point; e.g., the Fibonacci words
of length at least five are such. Also, it was shown there that
each binary word $w$ of length $|w| \ge 5$ and period
$\per(w) > |w|/2$ has less than $|w|/2$ critical points.

We shall now study the number of critical points in
ternary square-free words.
We show that, each sufficiently long square-free word $w$ 
can have at most $|w|-5$ critical points, and the bound $|w|-5$
is obtained by infinitely many square-free $w$.
Also, we prove that a square-free word $w$
has at least $|w|/4$ critical points, and that there is a sequence of 
square-free words closing to this bound.

\section{Preliminaries}\label{sec.prelim}

For a more extensive introduction to
combinatorics on words, including square-freeness and
critical factorisation,
we refer to Lothaire~\cite{LothaireI}.

For a finite alphabet $\Al$, let $\Al^\ast$ denote the monoid 
of all finite words over $\Al$ under concatenation.
The empty word is denoted by~$\eps$.
Let $w\in \Al^\ast$.
The length $|w|$ of~$w$ is the number of the occurrences of its letters. 
If $w=w_1 u w_2$ then $u$ is a~\emph{factor} of $w$.
It is a \emph{prefix} if $w_1=\eps$, and a \emph{suffix} 
if $w_2=\eps$.
The word~$w$ is said to be \emph{bordered} if there exists a nonempty word~$v$, with $v \ne w$, that is both a prefix and a suffix of $w$.

A word $w \in \Al^*$ is \emph{square-free}
if it has no factors of the form $vv$ for nonempty words $v$. 
Axel Thue~\cite{Thue:12} showed in 1912 that there are
square-free words over a ternary alphabet
$\Al_3=\{0,1,2\}$. One such word is obtained by iterating
the following morphism $\tau\colon \Al_3^* \to \Al_3^*$ on the initial letter $0$:
\[
\tau(0)=012, \quad \tau(1)=02, \quad \tau(2)=1\,.
\]
The iteration ultimately gives an infinite square-free word
\[
\Hall = 012 02 1 012 1 02 012 \cdots
\]
that does not contain the short words $010$, $212$ and $01201$
as its factors. The infinite word $\Hall$ is sometimes called a
\emph{variation of Thue-Morse word}; see~\cite{Blanchet-Sadri:integers}.

\begin{lemma}\label{lem:overlap}
Let $x$ be a nonempty factor of a square-free word $w$.
Then ``$x$ does not overlap with itself in $w$'', meaning that if
$w= u x_1 x_2 x_3v$ where $x=x_1x_2=x_2x_3$
and $x_2 \ne \eps$ then $x_1=\eps=x_3$.
\end{lemma}
\begin{proof}
Overlapping means, see e.g.~\cite{LothaireI}, that $x_1$ and $x_3$
are conjugates: $x_1=r s$, $x_3=sr$  and $x_2=(rs)^kr$ for some $r, s$
and $k \ge 0$.  But $x_1x_2x_3=(rs)^{k+2}r$ does contain a square
even if $k=0$.	
\end{proof}

\goodbreak
\section{Critical Factorisations}

We follow the main notations of~\cite{HarjuNowotka}.

An integer~$p$, with $1\leq p\leq |w|$, is a \emph{period}
of~$w$ if for the prefix $u$ of $w$ of length~$p$,
$w$ is a prefix of $u^n$ for some $n$.
The \emph{minimal period} of~$w$ is denoted by~$\per(w)$.
We have that $w$ is unbordered if and only if $\per(w)=|w|$.

An integer~$p$ with \hbox{$1\leq p < |w|$} is called 
a \emph{position} or a \emph{point} in~$w$. It denotes the place
after the prefix $x$ of length $p$: $w=x{\cdot}y$, $|x|=p$.
Thus there are $|w|-1$ positions in $w$.
A nonempty word $u$ is a \emph{repetition word} at~$p$
if there are words $x'$ and $y'$ (possibly empty) such that
$u=x'x$ or $x=x'u$, and $u=yy'$ or $y=uy'$.
If here $|u| > |x|$ (resp. $|u| > |y|$) then $u$ is said to have \emph{left overflow} (resp., \emph{right overflow}) at $p$;
see Fig.~\ref{Pic:reps}.

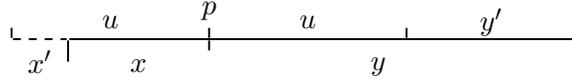
\begin{figure}[htb]
\tikzstyle{mystate}=[inner sep=0pt, outer sep=0pt]
\begin{center}
\begin{tikzpicture}[line width=0.06pc, scale=0.75](10,2)
\node  [mystate] (Z) at (0,0.5) {};
\node  [mystate] (A) at (1,0.5) {} ; 
\node  [mystate] (B) at (10,0.5) {};
\node [mystate]  (D) at (3.5,0.5) {};
\node  [mystate] (C) at (7,0.5) {};
\draw (A)--(B);
\draw [dashed] (Z)--(A);
\draw (Z)--(0,0.7);
\draw (A)--(1,0.1);
\draw (B)--(10,0.1);
\draw (D)--(3.5,0.3);
\draw (C)--(7,0.7);
\node at (1.75,0.8) {$u$};
\node at (5.25,0.8) {$u$};
\node at (6.5,0.02) {$y$};
\node at (0.5,0.1) {$x'$};
\node at (2.25,0.02) {$x$};
\node at (3.5,1.0) {$p$};
\node at (8.5,0.8) {$y'$};
\draw (D)--(3.5,0.7);
 \end{tikzpicture}
 \end{center}
 \caption{A repetition word $u$ of $w=xy$
having left overflow 
at position $p=|x|$.
}
 \label{Pic:reps}
\end{figure}

The length of a repetition word of $w$ at $p$ is 
called a \emph{local period} at~$p$.
The \emph{minimal local period} of $w$ at $p$ is denoted by
\[
\per(w,p) = \min\{ q \mid  q \text{ a local period of $w$ at } p \}.
\] 
Clearly, the (global) period $\per(w)$ is a local period at 
every point, and hence $\per(w,p) \le \per(w)$ for all $p$.
 A~position $p$ of $w$ is said to be \emph{critical} if $\per(w, p)=\per(w)$. 

The following result follows from the minimality assumption on $\per(w, p)$.

\begin{lemma}\label{lem:per}
A repetition word $u$ of $w$ at $p$ of length $\per(w,p)$
is unique and it is unbordered.
\end{lemma}

For a word $w$, we let
\[
\Cr(w) = \text{the number of critical points of $w$}.
\]
The number
\[
\frac{\Cr(w)}{|w|-1}
\]
is called the \emph{density} of the critical points in $w$.

\begin{example}
Let $w=0120201202021021021$ be an unbordered 
word of length 19, i.e., $\per(w)=|w|$. It is not square-free.
The minimal local periods of $w$ are in order of the 18 positions
\[ 
3, 5, 5, 2, 5, 5, 19, 19, 2, 2, 19, 19,3,3,3,3,3,3\,.
\]
In this example, $\Cr(w)=4$, and the density of critical points is
$4/18 = 0.222\dots$
\qed\end{example}

The Critical Factorisation Theorem is due to C\'esari
and {Vin\-cent}~\cite{CesariVincent:78}.
The present form of the theorem was developed by Duval \cite{Duval:79}; for the proofs,
see also ~\cite{CrochemorePerrin:91},~\cite{HarjuNowotka}
and Chapter~$8$ in~\cite{LothaireII}.

\begin{theorem}[Critical Factorisation Theorem]\label{CFT}
   Every word $w$ of length $|w|\geq 2$ has a critical point.
  Moreover, there is a critical point $p$ satisfying $p \le \per(w)$.
\end{theorem}

\EXTRA{
\begin{theorem}\label{thm.cft2}
   Each set of $\per(w)-1$ consecutive points in $w$, where
   $|w|\geq 2$, has a~critical point.
\end{theorem}
}

Later, in the statements of the results, we assume that
$|w| \ge 2$ to avoid the trivial exceptions.

\begin{lemma}\label{LRover}
Let $u$ be a repetition word of $w$ at $p$ of length $\per(w,p)$.
 If $u$  has both left and right overflows at $p$
then $p$ is a critical point.	
\end{lemma}

\begin{proof}
Let $w=x y$ where $u=x'x=yy'$ for nonempty words $x', y'$;
see Fig.~\ref{Pic:Reps2}.
By symmetry, we may assume that $|x'| \le |y|$ (otherwise
$|y'| \le |x|$).
Therefore $y=x'z$ and $x=z y'$ for some $z$.
Now, $w=x y= zy'x'z$, and hence $|zy'x'|$ is a period
of $w$, i.e.,  $\per(w) \le |zy'x'|$. But $|zy'x'|=|x'zy'|=|u|$ which shows that $\per(w,p)= |u|=\per(w)$
implying that $p$ is a critical point.
\end{proof}

\begin{figure}[htb]
\unitlength=5mm
\tikzstyle{mystate}=[inner sep=0pt, outer sep=0pt]
\begin{center}
\begin{tikzpicture}[line width=0.06pc](10,1)
\node  [mystate] (Z) at (0,0.5) {};
\node  [mystate] (A) at (2,0.5) {} ; 
\node  [mystate] (B) at (9,0.5) {};
\node[mystate] (D) at (8,0.5) {};
\node  [mystate] (C) at (4.5,0.5) {};
\draw (A)--(D);
\draw [dashed] (Z)--(A);
\draw [dashed] (D)--(B);
\draw (Z)--(0,0.7);
\draw (A)--(2,0.7);
\draw (B)--(9,0.7);
\draw (D)--(8,0.7);
\node at (1,0.8) {$x'$};
\draw (C)--(4.5,0.7);
\node at (3.25,0.8) {$x=zy'$};
\node at (8.5,0.8) {$y'$};
\node at (6.25,0.8) {$y=x'z$};
\node at (4.55,0.2) {$w$};
\draw (A)--(2,0.3);
\draw (D)--(8,0.3);
 \end{tikzpicture}
 \end{center}
 \caption{Left and right overflows imply criticality.}
 \label{Pic:Reps2}
\end{figure}
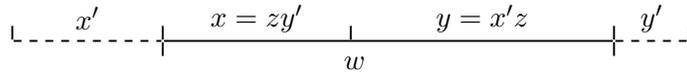

\section{Maximum number of critical points}

The next theorem follows from the observation that 
if a point $p$ of $w$ has neither left nor right overflow,
the minimal repetition word $u$ at $p$ supplies
a square $uu$ in~$w$.

\begin{theorem}\label{thm:over}
A word $w$ is square-free if and only if 
each repetition word at each position $p$ has left 
or right overflow, or both.
\end{theorem}

\begin{example}
The square-free word 
$w=0102 0120 2102 0102 1$ of length 17
is unbordered, i.e., $\per(w)=17$. It has 9 critical points
at the consecutive positions  $p=5, 6, \ldots, 13$.
This gives the density number $9/16\approx 0.56$.
For instance, the position $p=4$ has the minimal repetition word
$u=012021020102$, since
$u$ is the shortest factor
after the prefix $0102$ that ends with $0102$.  Thus $\per(w,4)=12$. 
\qed\end{example}

For a word $w$, let
\[
M(w)= \Big\lfloor \frac{|w|+1}{2} \Big\rfloor 
\]
denote the \emph{midpoint} of $w$. For odd length $|w|$,
it is just a choice of the two points nearest to the centre of $w$.

\begin{lemma}\label{lemma:mid}
For a square-free word $w \in \Al_3^*$, the position $M(w)$ is critical.
\end{lemma}
\begin{proof}
For even $|w|$, the claim is clear from Theorem~\ref{thm:over}.

Suppose then that $|w|=2k+1$, and let $u$ be the minimal local repetition word of $w$ at $M(w)=k+1$.
Suppose $u$ has right but not left overflow. Then $|u| = k+1$, and 
hence $w=v a v$ where $u=v a$ for a prefix $v$ and an overflow letter $a$.
But then $\per(w, k+1)=|u|=\per(w)$, and the claim follows.
\end{proof}

\begin{theorem}\label{thm:1}
The minimal local periods form
a unimodular sequence for square-free ternary words $w \in \Al_3^*$, i.e.,
\begin{align*}
&\per(w,p-1) \le \per(w,p) \ \text{ for } \ p \le M(w)\\
&\per(w,p) \le \per(w,p-1) \ \text{ for } \ p \ge M(w)\,.
\end{align*}
In particular,
the critical points $p$ of $w$ form an interval $q_0 \le p \le q_1$  
for some $q_1 \le M(w)$ and 
$q_2 \ge M(w)$.
\end{theorem}

\begin{proof}
Let $2\le p \le M(w)$.
The cases for $p \ge M(w)$ follow by considering the reverse of the word $w$ which is also square-free.
Let the minimal repetition word of $w$ at $p$ be $u$, i.e.,
$|u| = \per(w, p)$. Since $w$ is square-free and $|u| \ge 2$, 
$u$ has left overflow. If it also has right overflow then
$p$ is critical by Lemma~\ref{LRover}.
Let $a$ be the letter such that $u=v a$.
Then $|a v|$ is a local period at $p-1$ since the position $p-1$
has a repetition word $a v$. (It need not be minimal.)
Hence
$\per(w,p-1) \le \per(w,p)$.

\smallskip

For the second claim, by Lemma~\ref{lemma:mid}, $w$ has a critical point $p$ with $p\le M(n)$ and a critical point
$q \ge M(w)$. This proves the claim.
 \end{proof}

\begin{example}
Consider the prefix $w=\tau^5(0)$ of the square-free word $\Hall$,
i.e.,
\[ 
w=	012 02 1 012 1 02 012 02 1 02 012 1 .
\]
It is unbordered with $|w[=24$.
The sequence of the 23 minimal local periods is
\[
3,6,6,12,12,12,12, 24, \ldots, 24, 14,14,6,2 .
\]
Thus $\Cr(w)=12$, i.e., just over one half
of the positions are critical.
\qed\end{example}

\begin{theorem}\label{thm:low}
For each square-free ternary word $w$ of length $|w| \ge 26$,
we have $\Cr(w) \le |w|-5$. 
\end{theorem}

\begin{proof}
Let $w \in \Al_3^*$ be a square-free ternary word of length $n \ge 26$.
We show that has at least four non-critical points among the $n-1$ positions.
First of all, the points 1 and $n-1$ are non-critical,
since every letter of $\Al_3$ occurs in every factor of length four.


Without restriction we may assume that $01$ is a prefix of $w$.
Suppose that $w$ has at most three non-critical points.
Then the prefix $01$ occurs in $w$ at least twice. Indeed,
the word $\alpha=0121021202102$ of length 13 is 
uniquely the longest unbordered square-free word having $01$ only
as its prefix.
By Theorem~\ref{thm:1}, we can assume that 
the position $p=2$ is not critical; otherwise we consider the point $n-2$. 
Hence $w=01x01y$ where  $|01x|=\per(w)$ 
 and $y$ is a proper prefix of $x$.
Also, the prefix $01x$ is unbordered, since $01$ does not
occur in $x$ by the assumption that $p=2$ is not critical.
But now $|01x| \le |\alpha|=13$ and so $n\le 25$; a contradiction.
\end{proof}

\begin{example}
The word 
$w=01210212021020121021202$
of length 23 with $\per(w)=13$ has only three
non-critical points, $p=1,2,22$.
On the other hand, e.g.,
$v=01020121021201020121020$
of length 23 with $\per(v)=22$ has 14 non-critical points.
\qed\end{example}

The upper bound on the critical points is optimal:

\begin{theorem}
There are arbitrarily long square-free words $w \in \Al_3^*$
with $\Cr(w) =|w|-5$.	
\end{theorem}

\begin{proof}
We rely on the infinite square-free word $\Hall$ that is
a fixed point of the morphism $\tau$.
Consider the factors of $\Hall$
of the form $\beta=10201\alpha12021$. For our purpose, it  suffices to choose the words $\beta$ that start after the position 9 
of  $\Hall$, i.e., just after the prefix $012021012$.
There are infinitely many words $\beta$ since the suffix $12021$ is a factor of~$\tau^2(0)$.

For fixed middle word $\alpha$, consider $w=0\beta 2=010201\alpha 120212$ that begins and ends in the `forbidden' words $101$ and $212$
that do not occur in $\Hall$.
It is, clearly, square-free and unbordered. Each point $p$
with $2 < p < |w|-2$ is critical, since
the minimal repetition word at $p$ must have both left and right overflow
in order to leap over a factor $101$ or $212$;
see Lemma~\ref{LRover}.
Table~\ref{Table:Thue} lists the local periods and the minimal repetition words
for the remaining four (non-critical) points.
\end{proof}

\begin{table}[htb]
\begin{center}
\begin{tabular}{c|c|c}
$p$ & $\per(w, p)$ & Rep. word \\ \hline	
$1$ & 2 &    $10$\\
$2$ & 4 &    $0201$\\
$|w|-2$&4& $1202$\\
$|w|-1$&2& $21$
\end{tabular}
\caption{Local periods of non-critical points.}\label{Table:Thue}	
\end{center}
\end{table}

\section{Minimum number of critical points}

We now turn to the minimality problem of critical points
in square-free words.

\begin{theorem}
For each square-free word $w \in \Al_3^*$, we have $\Cr(w) \ge |w|/4$.
\end{theorem}

\begin{proof}	
Let $w \in \Al_3^*$ be a square-free word of length $|w|=n$.
We remind first that the middle point $M(w)$ is always critical in $w$.
We show that the distance between two non-critical points on the opposite
sides of the middle point is at least $n/4$.
The claim then follows from Theorem~\ref{thm:1}.

Assume, contrary to the claim, that $p$ and $q$ are non-critical points such that
\begin{equation}\label{eq:dist}
p < n/2 < q \ \text{ and } \  q-p < n/4\,.
\end{equation}
Let $u$ and $v$ be the minimal repetition words at $p$ and $q$, respectively.
Consequently, the word~$u$ has left overflow, and $v$ has right overflow.
Since $p$ and $q$ are on the opposite sides of the middle point,
$p \ge n/4$ and $q \le 3n/4$.
From $p \ge n/4$ it follows that $|u| > n/4$;
for otherwise $u u$ would be a factor in~$w$. 
Similarly $|v| \ge n/4$ and $q-|v| < n/2$. 
Since $q-p < n/4$, we have $p+|u| \ge q$, i.e.,\ the second occurrence of $u$ reaches over the position $q$. 
Similarly the first occurrence of $v$ starts before 
the position $p$; see Fig.~\ref{Fig:dist1}, where $|z|=q-p$.

\begin{figure}[htb]
\tikzstyle{mystate}=[inner sep=0pt, outer sep=0pt]
\begin{center}
\begin{tikzpicture}[line width=0.06pc, scale=0.75](40,100)
\node  [mystate] (A) at (0,1) {} ; 
\draw (A)--(0,1.2);
\node  [mystate] (B) at (1,1) {};
\draw(1,0.8)--(1,1.2);
\draw[dashed] (A)--(B);
\node  [mystate] (C) at (5,1) {};
\draw (C)--(5,1.2);
\node  [mystate] (D) at (3,1) {};
\draw (D)--(3,0.8);
\node  [mystate] (E) at (10,1) {};
\draw (E)--(10,1.2);
\node  [mystate] (F) at (8,1) {};
\draw (F)--(8,0.8);
\node  [mystate] (G) at (11,1) {};
\draw (11,1.2)--(11,0.8);
\node  [mystate] (H) at (13,1) {};
\draw (H)--(13,0.8);
\draw [dashed](G)--(H);
\draw [line width=1.25pt] (B)--(G);
\node at (0.5,1.3){$u_1$};
\node at (2.25,1.3) {$u_2$};
\node at (4.0,1.3) {$u_3$};
\node at (6.5,1.3) {$z$};
\node at (9,1.3) {$v_1$};
\node at (10.5,1.3) {$v_2$};
\node at (12.0,1.3) {$v_3$};
\draw (0,1.75)--(4.95,1.75);
\draw (5.05,1.75)--(10,1.75);
\node at (2.5,2.0) {$u$};
\node at (7.5,2.0) {$u$};
\draw (3,0.35)--(7.95,0.35);
\draw (8.05,0.35)--(13,0.35);
\node at (6,0) {$v$};
\node at (10,0) {$v$};
\draw [dashed] (C)--(5,2.25);
\draw [dashed] (A)--(0,1.75);
\draw [dashed] (D)--(3,0.35);
\draw [dashed] (E)--(10,1.75);
\draw [dashed] (F)--(8,0.0);
\draw [dashed] (H)--(13,0.35);
\node at (5,2.5) {$p$};
\node at (8,-0.25) {$q$};
\end{tikzpicture}
 \end{center}
 \caption{The local repetition words $u$ and $v$
for non-critical points $p$ and $q$.}
 \label{Fig:dist1}
\end{figure}
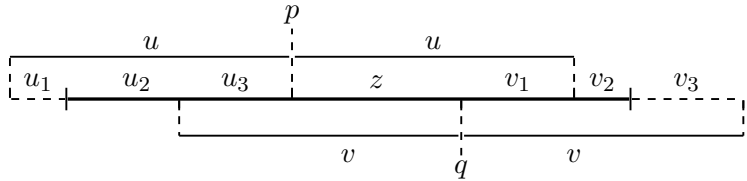

We now rely on the notations of the factors in Fig.~\ref{Fig:dist1}.

The words $u_3$ and $v_1$ are both prefixes of $v$ and suffixes of $u$.
If $|v_1| > |u_3|$ then, as prefixes of $v$, we have $v_1=u_3x$ for some nonempty $x$. But then $xx$ would be a square at $p$; a contradiction. 
If $|u_3| > |v_1|$ then, as suffixes of $u$, we have $u_3=xv_1$
for some nonempty $x$ yielding again a square $xx$ at $p$; a contradiction.
Therefore $v_1=u_3$. In this case $z=u_1u_2=v_2v_3$, and 
\[
w=u_2v_1 z v_1 v_2 = u_2v_1u_1u_2v_1v_2\,.
\]
Now, $v=v_1u_1u_2=v_1v_2v_3$ and so $u_1u_2=v_2v_3$
meaning that one of $u_1$ or $v_2$ is a prefix of the other.
To avoid $(u_2v_1u_1)^2$ in $w$, the word $v_2$ must be a proper prefix of $u_1$.
But now $\per(w) \le |u_2v_1u_1|=|u|=\per(w, p)$ contradicting the assumption
that $p$ was not critical.
This proves the claim.
\end{proof}

\EXTRA{
\begin{figure}[htb]
\unitlength=5mm
\tikzstyle{mystate}=[inner sep=0pt, outer sep=0pt]
\begin{center}
\begin{tikzpicture}[line width=0.06pc](40,100)
\node  [mystate] (A) at (2,1) {} ; 
\node  [mystate] (B) at (10,1) {};
\draw (A)--(B);
\node  [mystate] (Z) at (0,1) {};
\draw [dashed] (Z)--(A);
\draw (Z)--(0,1.2);
\draw (A)--(2,1.3);
\draw (A)--(2,0.7);
\draw (B)--(10,1.3);
\node  [mystate] (C) at (4.5,1) {};
\draw (C)--(4.5,1.2);
\node[mystate] (D) at (9,1) {};
\draw (D)--(9,1.2);
\node at (1,1.3) {$as$};
\node at (3.25,1.3) {$x$};
\node at (9.5,1.3) {$t$};
\node at (4.7,1.3) {$a$};
\node at (6.75,1.3) {$sx$};
\node at (6.75,1.75) {$y=sxt$};
\node [mystate] (X) at (1,1) {};
\draw (X)--(1,0.8); 
\draw[dashed] (1,0.9)--(2,0.9);
\node at (1.5,0.7) {$z$};
\draw (4.75,0.8)--(4.75,1);
\node [mystate] (Y) at (8.5,1) {};
\draw (Y)--(8.5,0.8);
\draw (D)--(9,0.8);
\node at (8.75,0.7) {$r$};
\node at (6.75,0.7) {$v=zxa$};
\end{tikzpicture}
 \end{center}
 \caption{$w=yaxyax_1$ with period $p=|yax|$.
The local squares are $axy$ and $zya=xy$,
respectively.
It is not assumed that $z$ is a suffix of $ax$.}
 \label{Pic:2}
\end{figure}
}

For the existence part of the next theorem, we take a quick
technical analysis of the prefixes of the word $\Hall$.
An induction argument gives $|\tau^n(0)|=3\cdot 2^{n-1}$,
$|\tau^n(1)|=2^{n}$ and $|\tau^n(2)|=2^{n-1}$.
For instance,
\[
|\tau^{n+1}(0)| = |\tau^n(012)| = 3\sdot 2^{n-1} + 2^n + 2^{n-1}
= 3\sdot 2^n\,.
\]

\newcommand{\pr}{\mathbf{p}}

Define the words $\Hall_n$, for $n \ge 1$, as follows 
\[
\Hall_n = \tau^{2n-1}(0) \tau^{2n-3}(0) \cdots \tau^3(0)\tau(0)\,.
\]
We show that $\Hall_n 0$ is a prefix of $\Hall$ of length $4^n$.
First $\Hall_1 0= 0120=\tau(0)0$ is a prefix of $\Hall$.
Inductively, we have
\[
\tau^2(\Hall_n 0) =
\tau^{2n+1}(0) \tau^{2n-1}(0) \cdots \tau^3(0)\tau^2(0)
= \Hall_{n+1}0\sdot 21\,.
\]
 and hence also $\Hall_{n+1}0$ is a prefix of $\Hall$.

For the length of $\Hall_n$, we obtain
\[
|\Hall_n|=\sum_{i=1}^{n} 3\cdot 2^{2(n-i)} = 
3\sum_{i=1}^{n} 4^{n-i} = 
4^n-1.
\]
As a prefix of $\Hall$, the word $\Hall_n$ is square-free. 

\begin{theorem}\label{thm:exists}
For all real numbers $\delta > 0$, there exists a square-free
ternary word $w=w(\delta)$ the density of which satisfies
\[
0.25 < \frac{\Cr(w)}{|w|} < 0.25+\delta\,.
\] 	
\end{theorem}

\begin{proof}
For any square-free word $x\in \Al_3^*$, let
\begin{equation}\label{eq:wx}
w_x= 0x02x10x02x0\,.
\end{equation}
Suppose first that $w_x$ is square-free, and thus that $x$ does not overlap with itself in~$w_x$. The suffix $2x0$ of $w_x$ does not occur elsewhere in $w_x$,
and hence the point $3|x[+6 $ is critical, since it must have both overflows.
It is the rightmost critical point. Indeed, $\per(w_x,3|x|+7)=|x|+2$.
For the point $t=2|x|+3$, the minimal repetition word is
$10x02x$ of length $2|x|+4 < \per(w)$ since $\per(w) > 3|x|+7$.
Hence the middle point $2|x|+4$ is the leftmost critical point.
It follows that $w_x$ has $q-p+1=|x|+3$ critical
points. Thus  
\[
\frac{\Cr(w)}{|w_x|} = \frac{|x|+3}{4|x|+8}= 0.25 + \frac{1}{|w_x|}\,,
\]
which has the limit $0.25$ as $|x| \to \infty$.

It remains to show that there are arbitrarily long square-free words $x$
for which $w_x$ is square-free. 
Again, we lean on the word
$\Hall$. 
We consider the words $w_{x_n}$ where 
\[
x_n = 120102\, \Hall_n\,.
\]
We have
\begin{align*}
w_{x_n}= \,
& 0\sdot120102 \Hall_n\cdot 02\sdot 120102\Hall_n\sdot 1
0\sdot120102 \Hall_n\cdot 02\sdot 120102 \Hall_n\sdot 0
\end{align*}
Since $010$ and $212$ do not occur in $\Hall$, both $010$ and $212$ 
would have to be aligned in any square $u u$ of $w_{x_n}$,
which is not possible by the `markers' $02$, $1$ and $0$ dividing the word.
Also, since $\Hall_n$ has a border $\tau(0)$,
one easily checks that there are no short squares $u u$ in $w_x$
for $|u| \le 4$.
Hence a possible square must be inside one of the words 
(a) $102\Hall_n021$,
(b) $102\Hall_n 101201$, or
(c) $102\Hall_n 0$.
We consider these cases separately.
Recall that $\Hall_1=012=\tau(0)$. Also, since  $\Hall$ is a fixed point
of the morphism $\tau$, whenever $v$ is a factor of
$\Hall$, so is $\tau(v)$.

\smallskip
(a) 
Let $\alpha_n=102\Hall_n021$.
The word $\alpha_1=102012021$ occurs in $\Hall$ after position~9. 
We prove by induction that each $\alpha_n$ is a factor of $\Hall$,
and thus they are square-free.
Suppose, using \eqref{eq:Hall-n}, that 
\[
\alpha_i =102\Hall_i 021= 102 \tau^{2i-1}(0) \cdots \tau(0) 021
\]
is a factor of $\Hall$. Then
\[
\tau(\alpha_i) = 0201\sdot 21 \tau^{2i}(0) \cdots \tau^2(0) 012\sdot021\,
\]
where the indicated factor is denoted by 
$z=21 \tau^2(0) 012$.
By mapping with $\tau$, we obtain
\[
\tau(z)=102  \tau^{2i+1}(0) \cdots \tau^3(0) \tau(0) 021
= 102 \Hall_{i+1} 021 = \alpha_{i+1}.
\]
Hence $\alpha_n$ is a factor of $\Hall$ for all $n$. 

\smallskip
(b)
We employ in this case the same techniques as in (a) except that we need to
eliminate the last letter $1$ of the word. 
In order for $102\Hall_n 101201$ to have a square $u u$,
the former occurrence of $u$ in the square must be a factor of $102\Hall$.
However, $\Hall$ does not have a factor $01201$ since it would have to be
part of the square $012012$. Therefore we can, and must,  choose
$\beta_n=102\Hall_n 101202$.

The first occurrence of 
$\beta_1=102\Hall_1 101202 = 102\tau(0)101202$ 
in $\Hall$ starts after position~17.
We proceed inductively as in case (a). Suppose that
\[
\beta_i=102\Hall_i 101202 =
102\tau^{2i-1}(0) \cdots \tau(0) 101202
\] is a factor of $\Hall$. 
Mapping by $\tau$ gives
\[
\tau(\beta_i) = 0201\sdot 21 \tau^{2i}(0) \cdots \tau^2(0) 0201\sdot 202\,,
\]
where the indicated portion $z=21 \tau^{2i}(0) \cdots \tau^2(0) 0201$ gives
\[
\tau(z) = 102 \tau^{2i+1}(0) \cdots \tau^3(0) \tau(0)  101202
= 102 \Hall_{i+1}101202=\beta_{i+1}\,.
\]
Hence $\beta_n$ is a factor of $\Hall$, and thus square-free, for all $n$.

\smallskip
(c)
The word $102\Hall_n 0$ is a factor of $\alpha_n$ and thus square-free.
This proves the claim.
\end{proof}

The chosen words $x_n = 120102\, \Hall_n$ are not the only ones that
give a square-free word $w_x$.

\begin{problem}
Does there exist, for all $n$, a word $x$ of length $n \ge 20$ such that $w_x$? 
\end{problem}

\begin{problem}
Does there exist a word $w$ such that $\Cr(w)=|w|/4$?

\end{problem}

\bibliographystyle{plain}
\bibliography{bibo.bib}
\end{document}